\documentclass[12pt,oneside]{amsart}

\usepackage{amsthm}
\usepackage{mathrsfs}
\usepackage{enumerate}
\usepackage{amsfonts}
\usepackage{verbatim}
\usepackage{amsbsy}
\usepackage{amsmath}
\usepackage{amssymb}
\usepackage{MnSymbol}
\usepackage{mathrsfs}
\usepackage{enumitem}
\usepackage{cite}
\usepackage{graphicx}
\usepackage{tikz}
\usetikzlibrary{arrows}
\usepackage{multirow}
\usepackage{mathtools}

\newtheorem{theorem}{Theorem}[section]
\newtheorem{lemma}[theorem]{Lemma}

\newtheorem{proposition}[theorem]{Proposition}

\newtheorem{question}[theorem]{Question}
\newtheorem{corollary}[theorem]{Corollary} 

\theoremstyle{definition}
\newtheorem{definition}[theorem]{Definition}

\theoremstyle{remark}
\newtheorem{remark}[theorem]{Remark}
\newtheorem{example}[theorem]{Example}
\newtheorem{case}{Case}

\numberwithin{subcase}{case}

\makeatletter
\@addtoreset{case}{theorem}
\makeatother

\makeatletter
\def\@seccntformat#1{%
  \protect\textup{\protect\@secnumfont
    \ifnum\pdfstrcmp{subsection}{#1}=0 \bfseries\fi
    \csname the#1\endcsname
    \protect\@secnumpunct
  }%
}  
\makeatother

\makeatletter
\def\subsection{\@startsection{subsection}{3}%
  \z@{.7\linespacing\@plus.7\linespacing}{0.5\linespacing}%
  {\normalfont\bfseries}}
\makeatother

\DeclareMathOperator{\sym}{Sym}
\DeclareMathOperator{\mi}{R}
\DeclareMathOperator{\SP}{SP}


\begin{document}
\title{Counting Subwords in Circular Words and Their Parikh Matrices}

\author{Ghajendran Poovanandran}
\address{School of Mathematics, Actuarial and Quantitative Studies\\
	Asia Pacific University of Technology \& Innovation\\
	Technology Park Malaysia, Bukit Jalil\\
	57000 Kuala Lumpur, Malaysia}
\email{ghajendran@staffemail.apu.edu.my}

\author{Jamie Simpson}
\address{Department of Mathematics and Statistics\\
	Curtin University of Technology\\
	GPO Box U1987\\
	Perth, Western Australia 6845\\
	Australia}
\email{simpson@maths.curtin.edu.au}

\author{Wen Chean Teh$^*$}\thanks{$^*$Corresponding author} 
\address{School of Mathematical Sciences\\
	Universiti Sains Malaysia\\
	11800 USM\\  \linebreak Malaysia}
\email{dasmenteh@usm.my}

\begin{abstract}
The word inference problem is to determine languages such that the information on the number of occurrences of those subwords in the language can uniquely identify a word. A considerable amount of work has been done on this problem, but the same cannot be said for circular words despite growing interests on the latter due to their applications---for example, in splicing systems. Meanwhile, Parikh matrices are useful tools and well established in the study of subword occurrences. In this work, we propose two ways of counting subword occurrences in circular words. We then extend the idea of Parikh matrices to the context of circular words and investigate this extension. Motivated by the word inference problem, we study ambiguity in the identification of a circular word by its Parikh matrix. Accordingly, two rewriting rules are developed to generate ternary circular words which share the same Parikh matrix.
\end{abstract}

\maketitle

\section{Introduction}
The word inference problem is an actively researched topic in combinatorics on words \cite{RM17,jM00}. The problem is to determine an optimal set of subwords that, together with their multiplicity, can uniquely identify a word. This problem was first studied in \cite{MMSSS91} where the maximum length of a word which could be determined by its $k$-spectrum (i.e.~multiset of all non-contiguous subwords of length at most $k$) was investigated. Some other works on $k$-spectra include \cite{DS03,DFM19,SA19}.

However, considering all words up to a certain length is not desirable, especially when it is possible for a few subwords of different lengths to determine a word uniquely. Among innovations made to investigate this problem is the Parikh matrix mapping which was introduced in \cite{MSSY01}. The entries of a Parikh matrix are counts of occurrences of a certain set of subwords and a classical problem in this area--the injectivity problem--is to determine to what extent a word is determined by its Parikh matrix.
Due to their intrinsic usefulness in studying subword occurrences, Parikh matrices are well studied in the literature (for example, see \cite{SS06,aS06,aS10,GT17b,SJ20, poovanandran2020ambiguity, vS09}). There have recently been graph theoretic studies related to Parikh matrices \cite{teh2020parikh, poovanandran2018elementary}, suggesting a potential direction in this area that creates link between combinatorics on words and graph theory. 

Circular words, also known as necklaces or cyclic words in the literature, are different from traditional linear words---the former have neither a beginning nor an end. Circular sequences are not purely theoretical as they exist naturally in the DNA strands of certain viruses and bacteria \cite{HC71}. However, circular words have not been investigated as widely as linear words.  Some current active research directions pertaining to circular words are pattern avoidance\cite{fS05,CF02,sA10} and splicing systems\cite{SSD92,BDMZ05, BDZ10}. Until now, to our best knowledge, the work closest to the study of subword occurrences in circular words is \cite{sJ14}.  

In this work, we propose two different ways of counting the number of occurrences of subwords in circular words, which can be described as the direct and the average approach. The direct approach aligns with the notion of subword histories\cite{MSY04} while the average approach is compatible with an  extension  of Parikh matrices to circular words. The latter motivates the rest of this paper where we study our newly introduced Parikh matrices of circular words. In most cases, we restrict our attention to the properties of Parikh matrices of binary and ternary circular words. By our definition, as in the case of linear words, two or more circular words may share the same Parikh matrix. In the spirit of characterizing such words, two rewriting rules analogous to the ones introduced in \cite{AAP08} are presented.

The remainder of this paper is structured as follows. Section 2 provides the basic terminology and preliminaries. In Section 3, we propose, illustrate and study two different approaches to counting subword occurrences in circular words. Section 4 provides some historical background before introducing Parikh matrices of circular words, which are compatible with the average approach of counting subword occurrences. Our main theorem in this section shows that for the binary alphabet, the Parikh matrix of a circular word depends only on the Parikh vector of the word. Section 4 studies the ambiguity of Parikh matrices for circular words. Accordingly, for the ternary alphabet, we present two ways of rewriting a word without altering its Parikh matrix. Our conclusion follows after that.

\section{Preliminaries}
The cardinality of a set $A$ is denoted by $|A|$. For a matrix $X$, we denote its $(i,j)$-entry by $X_{i,j}$.

Suppose $\Sigma$ is a finite non-empty alphabet. The set of all words over $\Sigma$ is denoted by $\displaystyle{\Sigma^\ast}$. The unique empty word is denoted by $\lambda$. Given a word $w=a_1a_2\cdots a_n$ (where $a_i\in\Sigma$ for all $1\le i\le n$), we denote by $\mi(w)$ the \textit{reverse} of $w$, that is $\mi(w)=a_na_{n-1}\cdots a_1$. For every $w\in\displaystyle{\Sigma^\ast}$, $|w|$ denotes the length of $w$ and $w[i]$ denotes the letter in the $i^{th}$ position of $w$. A word $w\in\displaystyle{\Sigma^\ast}$ is \textit{primitive} if $w$ cannot be written as $v^k$ for any $v\in \displaystyle{\Sigma^\ast}$ and integer $k\ge 2$. Given two words $v,w\in \displaystyle{\Sigma^\ast}$, the concatenation of $v$ and $w$ is denoted by  $vw$. 

An \emph{ordered alphabet} is an alphabet $\Sigma=\{a_1, a_2, \dotsc,a_s\}$ with a total ordering on it. If $a_1<a_2<\dotsb < a_s$, then we may write $\Sigma= \{a_1<a_2<\dotsb<a_s\}$. 
The projective morphism over an alphabet is defined as follows:
\begin{definition}\label{DefProjection}
Suppose $\Sigma$ is an alphabet and $\Gamma\subseteq\Sigma$. The \textit{projective morphism} $\pi_{\Gamma}: \displaystyle{\Sigma^\ast}\rightarrow \displaystyle{\Gamma^\ast}$ is defined by 
\begin{equation*}
\pi_{\Gamma}(a)=\begin{cases}
a, & \text{if } a\in\Gamma\\
\lambda, & \text{otherwise.}
\end{cases}
\end{equation*}
\end{definition}

A word $v$ is a \emph{scattered subword} (in this paper, we would simply refer to it as \emph{subword}) of $w\in \displaystyle{\Sigma^\ast}$ if there exist $x_1,x_2,\dotsc, x_k$, $y_0, y_1, \dotsc,y_k\in \displaystyle{\Sigma^\ast}$ such that $v=x_1x_2\dotsm x_k$ and  $w=y_0x_1y_1x_2y_2\dotsm x_ky_k$. Note that traditionally, the term subword (without ``scattered") coincides with the term factor, which denotes a contiguous part of a word. The number of occurrences of a word $v$ as a subword of $w$ is denoted by $\vert w\vert_v$. Two occurrences of $v$ are considered different if and only if they differ by at least one position of some letter. For example, $\vert bcbcc\vert_{bc}=5$ and $\vert aabcbc\vert_{abc}=6$.
By convention, $\vert w\vert_{\lambda}=1$ for all $w\in \displaystyle{\Sigma^\ast}$. Let $\sym(k)$ denote the set of permutations on the set of integers $\{1,2,\ldots ,k\}$. The following is a basic combinatorial property of words. A brief proof is provided for completeness.
\begin{proposition}\label{CombPropLinWord}
	Suppose $\Sigma=\{a_1,a_2,\ldots ,a_s\}$ and $w\in \displaystyle{\Sigma^\ast}$. 
	Then 
	$$\sum\limits_{\sigma\in\sym(s)}|w|_{a_{\sigma(1)}a_{\sigma(2)}\cdots a_{\sigma(s)}}= \prod\limits_{i=1}^{s}|w|_{a_i}.$$ 
\end{proposition}

\begin{proof}
Fix an arbitrary word $w\in \displaystyle{\Sigma^\ast}$. The result holds trivially if $\vert w\vert_{a_i}=0$ for some $1\leq i\leq s$.	Assume otherwise.
Consider the collection of $s$-tuples $(j_1, j_2, \dotsc, j_s)$
such that $1\leq j_1< j_2 < \dotsb < j_s\leq \vert w\vert $ and $\vert w[j_1]w[j_2]\dotsm w[j_s] \vert_{a_i}=1$ for all $1\leq i\leq s$.
By the multiplication principle of counting, the number of such tuples is $\prod\limits_{i=1}^{s}|w|_{a_i}$. On the other hand, each such tuple specifies an occurrence of $a_{\sigma(1)}a_{\sigma(2)}\cdots a_{\sigma(s)}$ in $w$ for a unique $\sigma\in \sym(s)$, where
$a_{\sigma(k)}= w[j_k]$ for all $1\leq k\leq s$, and vice versa. Therefore, the number of such tuples is also given by $\sum\limits_{\sigma\in\sym(s)}|w|_{a_{\sigma(1)}a_{\sigma(2)}\cdots a_{\sigma(s)}}$.
\end{proof}

For any integer $k\geq 2$, let $\mathcal{M}_k$ denote the multiplicative monoid of $k \times k$ upper triangular matrices with nonnegative integer entries and units on the main diagonal.

\begin{definition} \cite{MSSY01}\label{240821}
Suppose $\Sigma=\{a_1<a_2<\cdots <a_s\}$ is an ordered alphabet. The \textit{Parikh matrix mapping} with respect to $\Sigma$, denoted by $\Psi_\Sigma$, is the morphism
\begin{equation*}
\Psi_\Sigma:\Sigma^{\ast}\rightarrow\mathcal{M}_{s+1},
\end{equation*}
defined as follows: $\Psi_{\Sigma}(\lambda)=I_{s+1}$; for each $1\leq q\leq s$, 	
$\Psi_{\Sigma}(a_q)=M$, where
\begin{itemize}
	\item $M_{i,i}=1$ for all $1\leq i\leq s+1$;
	\item $M_{q,q+1}=1$;
	\item all other entries of $M$ are zero; and	
\end{itemize}
for every $w=a_{\ell_1}a_{\ell_2}\dotsm a_{\ell_{\vert w\vert}}\in  \displaystyle{\Sigma^\ast}$, we have $\Psi_{\Sigma}(w)= \Psi_{\Sigma}( a_{\ell_1}) \Psi_{\Sigma}( a_{\ell_2}) 
\dotsm \Psi_{\Sigma}( a_{\ell_{\vert w\vert}})$.	
Matrices of the form $\Psi_\Sigma(w)$ for $w\in  \displaystyle{\Sigma^\ast}$ are called  \textit{Parikh matrices}.
\end{definition}

\begin{theorem}\cite{MSSY01}\label{1206a}
Suppose $\Sigma=\{a_1<a_2< \dotsb<a_s\}$ is an ordered alphabet and $w\in \displaystyle{\Sigma^\ast}$. The matrix $\Psi_{\Sigma}(w)=M$, has the following properties:
\begin{itemize}
\item $M_{i,i}=1$ for each $1\leq i \leq s+1$;
\item $M_{i,j}=0$ for each $1\leq j<i\leq s+1$;
\item $M_{i,j+1}=\vert w \vert_{a_ia_{i+1}\dotsm a_j}$ for each $1\leq i\leq j \leq s$.
\end{itemize}
\end{theorem}

\begin{remark}\label{010721}
Suppose $\Sigma=\{a_1<a_2< \dotsb<a_s\}$. The \emph{Parikh vector} $\Psi(w)=(|w|_{a_1},|w|_{a_2},\ldots,|w|_{a_s})$ of a word $w\in\Sigma^*$ is contained in the second diagonal of the Parikh matrix $\Psi_\Sigma(w)$.
\end{remark}

\begin{example}
	Suppose $\Sigma=\{a<b<c\}$ and $w=bacbc$, then:
	\begin{align*}
		\Psi_{\Sigma}(w)&=\Psi_{\Sigma}(b)\Psi_{\Sigma}(a)\Psi_{\Sigma}(c)\Psi_{\Sigma}(b)\Psi_{\Sigma}(c)\\
		&=\begin{pmatrix}
			1 & 0 & 0 & 0 \\
			0 & 1 & 1 & 0\\
			0 & 0 & 1 & 0\\
			0 & 0 & 0 & 1
		\end{pmatrix}
		\begin{pmatrix}
			1 & 1 & 0 & 0 \\
			0 & 1 & 0 & 0\\
			0 & 0 & 1 & 0\\
			0 & 0 & 0 & 1
		\end{pmatrix}
		\dotsm
		\begin{pmatrix}
			1 & 0 & 0 & 0 \\
			0 & 1 & 0 & 0\\
			0 & 0 & 1 & 1\\
			0 & 0 & 0 & 1
		\end{pmatrix}\\
		&=\begin{pmatrix}
			1 & 1 & 1 & 1 \\
			0 & 1 & 2 & 3\\
			0 & 0 & 1 & 2\\
			0 & 0 & 0 & 1
		\end{pmatrix}
		=\begin{pmatrix}
			1 & \vert w\vert_a & \vert w\vert_{ab} & \vert w\vert_{abc} \\
			0 &1 & \vert w\vert_b & \vert w\vert_{bc}\\
			0 & 0 & 1 & \vert w\vert_c\\
			0 & 0 & 0 & 1
		\end{pmatrix}.
	\end{align*}
\end{example}

\section{Counting Subword Occurrences in Circular Words}
In contrast to classical linear words, a \textit{circular word} (sometimes referred to as a \textit{necklace} in the literature) has neither a beginning nor an end. To describe it formally, we first need the following relation.
\begin{definition}
	Suppose $\Sigma$ is an alphabet and  $w=a_1a_2\cdots a_n\in \displaystyle{\Sigma^\ast}$. 
	For any integer $0\le i\le n-1$, the \textit{i-th cyclic shift} of $w$ is the word $w'=a_{i+1}a_{i+2}\cdots a_na_1\cdots a_i$. Two words $w,w'\in \displaystyle{\Sigma^\ast}$ are \textit{conjugates}, denoted by $w\sim_cw'$, if $w'$ is a cyclic shift of $w$.
\end{definition}
Note that although there are $n$ possible cyclic shifts of a word of length $n$, some of them may be equal---thus the number of conjugates of the word may be less than $n$. The conjugacy relation is clearly an equivalence relation over $\displaystyle{\Sigma^\ast}$. Thus, a circular word can be defined as follows.
\begin{definition}
	Suppose $\Sigma$ is an alphabet. The \textit{circular word over $\Sigma$} represented by a word $w\in \displaystyle{\Sigma^\ast}$, denoted by $[w]$, is the equivalence class of $w$ under the conjugacy relation. We denote the set of all circular words over $\Sigma$ by $\displaystyle{\Sigma_c^\ast}$.
\end{definition}
\begin{remark}\leavevmode 
	\begin{enumerate}
		\item The definition of a circular word $[w]$ is independent of the representative of $[w]$---the word $w$ is simply an arbitrary choice from the conjugacy class.
		\item If $w$ is a primitive word, then the cardinality $|[w]|$ of the conjugacy class $[w]$ equals $|w|$; otherwise if $w=v^k$ for some primitive word $v$ and integer $k\ge 2$, then $|[w]|=|v|$.
	\end{enumerate}
\end{remark}
Since a circular word $[w]$ is a class containing all cyclic shifts of $w$, we can view a circular word literally as a word drawn on a circle.

\begin{example}\label{ExCyclSeq}
	Suppose $\Sigma=\{a,b,c\}$ and consider the word $w=cabacb$. In order, $w$ and its cyclic shifts constitute the 
	circular word $[w]$:
	\begin{equation*}
[w]=\{cabacb,abacbc,bacbca,acbcab,cbcaba,bcabac\}.
	\end{equation*}
	We can thus write (in clockwise orientation for convention) the circular word $[w]$ as follows:
	\begin{center}
		\begin{tikzpicture}[>=stealth',semithick,auto,]
			[scale=.8,auto=left,every node/.style={}]
			\node at (0,0) {c};
			\node at (0.5,-0.3) {a};
			\node at (0.5,-0.8) {b};
			\node at (0,-1.1) {a};
			\node at (-0.5,-0.8) {c};
			\node at (-0.5,-0.3) {b};
		\end{tikzpicture}
	\end{center} 
\end{example}

Due to the cyclic structure of a circular word, the classical way of counting subword occurrences (as in the case of linear words) is not applicable. Thus, we propose and study two possible ways to count, for a circular word, the number of occurrences of a word as a subword in it.

\subsection{A Direct Approach}
We first provide some examples to illustrate our first approach. These will serve as a motivation for the definition that follows.

\begin{example}\label{ExTechnique1}
	Suppose $\Sigma=\{a,b,c\}$ and consider the circular word $[w]=[cabacb]$. We write $[w]$ as follows:
	\begin{center}
		\begin{tikzpicture}[>=stealth',semithick,auto,]
			[scale=.8,auto=left,every node/.style={}]
			\node at (0,0) {$c_1$};
			\node at (0.5,-0.3) {$a_1$};
			\node at (0.5,-0.8) {$b_1$};
			\node at (0,-1.1) {$a_2$};
			\node at (-0.5,-0.8) {$c_2$};
			\node at (-0.5,-0.3) {$b_2$};
		\end{tikzpicture}
	\end{center} 
	with the subscripts assigned to distinguish identical letters. We count four occurrences of $abc$ as a subword in $[w]$---particularly, $a_1b_1c_2$, $a_1b_1c_1$, $a_1b_2c_1$ and $a_2b_2c_1$. 
	Note that subwords are not allowed to overlap themselves. Hence, in this example, $a_1b_2c_2$ is not a subword. Also, notice that the count of the distinct combinations of letters also corresponds to the sum $\sum\limits_{u\in[abc]}|w|_u=|w|_{abc}+|w|_{bca}+|w|_{cab}=1+0+3=4$, where $w$ is the linear word $cabacb$.
\end{example}

\begin{example}\label{ExTechnique2}
	Consider the circular word $[w]=[aaaaaa]$. We write $[w]$ as follows:
	\begin{center}
		\begin{tikzpicture}[>=stealth',semithick,auto,]
			[scale=.8,auto=left,every node/.style={}]
			\node at (0,0) {$a_1$};
			\node at (0.5,-0.3) {$a_2$};
			\node at (0.5,-0.8) {$a_3$};
			\node at (0,-1.1) {$a_4$};
			\node at (-0.5,-0.8) {$a_5$};
			\node at (-0.5,-0.3) {$a_6$};
		\end{tikzpicture}
	\end{center} 
	with the subscripts assigned to distinguish identical letters. We count $\binom{6}{2}=15$ occurrences of $aa$ as a subword in $[w]$. (Note that $a_ia_j$ and $a_ja_i$ are considered as the same occurrence.) Also, $\sum\limits_{u\in[aa]}|w|_u=|w|_{aa}=\binom{6}{2}=15$.
\end{example}

Generally, suppose $\Sigma$ is an alphabet and $[w]$ is a circular word over $\Sigma$ with $w$ as a fixed representative. Let $v\in\displaystyle{\Sigma^\ast}$. Motivated by the previous two examples, the following is our first proposed count of occurrences of  $v$ as a subword of $[w]$:
\begin{align*}
&\left\vert \left\{ \, (i_1,i_2, \dotsc, i_{|v|})\in \mathbb{N}^{|v|} \mid 1\leq i_1<i_2<\dotsb<i_{|v|}\leq \vert w\vert  \text{ and }   w[i_1]w[i_2]\dotsm w[i_{|v|}]\in [v]\, \right\}     \right|\\
={}& \left\vert \bigcup_{u\in [v]}  \left\{\,  (i_1,i_2, \dotsc, i_{|v|})\in \mathbb{N}^{|v|}\mid 1\leq i_1<i_2<\dotsb<i_{|v|}\leq \vert w\vert  \text{ and }   w[i_1]w[i_2]\dotsm w[i_{|v|}] =u \,\right\}       \right\vert\\
={}& \sum_{u\in [v]} \left\vert  \left\{ \, (i_1,i_2, \dotsc, i_{|v|})\in \mathbb{N}^{|v|} \mid 1\leq i_1<i_2<\dotsb<i_{|v|}\leq \vert w\vert  \text{ and }   w[i_1]w[i_2]\dotsm w[i_{|v|}] =u \,\right\}       \right\vert.
\end{align*}
Therefore, it leads to our following definition.

\begin{definition}\label{DefDirect}
Suppose $\Sigma$ is an alphabet and $[w]\in  \displaystyle{\Sigma_c^\ast}$. 
The number of occurrences of a word $v\in\Sigma^*$ as a subword of $[w]$, denoted by $|[w]|_v$ is defined by $|[w]|_v=\sum\limits_{u\in[v]}|w|_u$.
\end{definition}

\begin{remark}
Based on our intended interpretation and discussion above, the value $|[w]|_v$ does not depend on the representative of the circular word  (in this case, $w$) as for any word $w'$ with $w'\sim_c w$, we have $\sum\limits_{u\in[v]}|w|_u=\sum\limits_{u\in[v]}|w'|_u$. This can also be shown directly from definition~\ref{DefDirect}. 
\end{remark}

\begin{example}
Suppose $\Sigma=\{a,b,c\}$ and consider the two conjugate words $w=cabacb$ and $w'=bacbca$. We have 
\begin{align*}
|[w]|_{abc}&=|w|_{abc}+|w|_{bca}+|w|_{cab}\\
		   &=1+0+3=4;\\
|[w']|_{abc}&=|w'|_{abc}+|w'|_{bca}+|w'|_{cab}\\
		   &=1+3+0=4.
\end{align*}
\end{example}

\begin{remark}
Suppose $\Sigma$ is an alphabet, $[w]\in \displaystyle{\Sigma_c^\ast}$. If $v=a^k$ for some $a\in\Sigma$ and integer $1\le k\le |w|_a$, then $|[w]|_v=\binom{|w|_a}{k}$.
\end{remark}

We now consider the occurrence of a particular class of linear words, introduced in \cite{GT17b}, as subwords in circular words.

\begin{definition}\label{010921d}
Suppose $\Sigma$ is an alphabet. A word $u\in\displaystyle{\Sigma^\ast}$ is a \textit{slender Parikh word} if  $|u|_a\le 1$ for all $a\in\Sigma$.
Let $\SP_\Sigma$ denote the set of all slender Parikh words over $\Sigma$.
\end{definition}
Let $\Sigma$ be an alphabet and $s= \vert \Sigma\vert$. There are  $s!$ slender Parikh words of length $s$ over $\Sigma$ and these
can be partitioned into $(s-1)!$ equivalence classes over the conjugacy relation. The representatives of those equivalence classes can be used to extend Proposition~\ref{CombPropLinWord} to the context of circular words.

\begin{proposition}\label{ProSPandCirc}
	Suppose $\Sigma=\{a_1<a_2<\cdots <a_s\}$ and $[w]\in \displaystyle{\Sigma_c^\ast}$. Let $v_i, \,1\le i\le (s-1)!$ enumerate arbitrary representatives, one from each of the $(s-1)!$ conjugacy equivalence classes constituting the words in $\SP_{\Sigma}$ of length $s$. Then
	$$\sum_{1\le i \le (s-1)!}|[w]|_{v_i}=\prod\limits_{i=1}^{s}|w|_{a_i}.$$
\end{proposition}
\begin{proof}
	$\displaystyle\sum_{1\le i\le (s-1)!}|[w]|_{v_i}=
	\sum_{1\le i\le (s-1)!}\sum_{u\in[v_i]}|w|_{u}=
	\sum\limits_{\sigma\in\sym(s)}|w|_{a_{\sigma(1)}a_{\sigma(2)}\cdots a_{\sigma(s)}}=
	\prod\limits_{i=1}^{s}|w|_{a_i}.$
\end{proof}

Definition~\ref{DefDirect} can be restated using a known terminology. Subword histories over an alphabet $\Sigma$ were introduced by Mateescu et al.~\cite{MSY04} and they are polynomials with integer coefficients in terms of words over $\Sigma$, for example, $abc \times bc - 2ab+c$. They are meant to study identities or inequalities involving counts of subword occurrences. For that purpose, every subword history is associated to an evaluation. As an illustration, the value of the aforementioned subword history in a word  $w$ is defined to be $|w|_{abc}|w|_{bc}-2|w|_{ab}+|w|_c$.

The introductory paper \cite{MSY04}, partly motivated by Parikh matrices, studied certain associated canonical subword histories. Since the count $|[w]|_v$ of subword occurrences in the circular word $[w]$ as defined in Definition~\ref{DefDirect} is the same as the value of the subword history $\sum_{u\in [v]} u$, it suggests that subword histories of the form  $\sum_{u\in [v]} u$ are also worth studying.
 
\subsection{An Average Approach}
Our second way of counting subword occurrences in circular words is based on the view that the count is the average of the ones obtained for each linear word in the associated conjugacy equivalence class.

\begin{definition}\label{DefParikhCircular}
Suppose $\Sigma$ is an alphabet and $[w]\in \displaystyle{\Sigma_c^\ast}$. The number of occurrences of a word $v\in \displaystyle{\Sigma^\ast}$ as a subword of $[w]$ is defined by $|[w]|_v=\frac{1}{|[w]|}\sum\limits_{u\in[w]}|u|_v$.
\end{definition}

We abuse notation and let $\vert [w]\vert_v$ denote the count in both the direct approach and the average approach.
However, from now onwards, $\vert [w]\vert_v$ always refers to the count as defined in Definition~\ref{DefParikhCircular} and thus less ambiguity is present. Furthermore, the remainder of this work can be interpreted as providing an argument that the average approach is desirable over the direct one in the context being studied.

\begin{remark}\label{AltDef}
	Alternatively, the definition of $|[w]|_v$ can also be expressed in terms of the cyclic shifts of $w$. 
	For every integer $0\leq i\leq \vert w\vert-1$, let $w_i$ denote the $i$-th cyclic shift of $w$. Then for every $v\in \displaystyle{\Sigma^\ast}$, 
	$$\vert [w]\vert_v = \frac{1}{\vert w\vert}\sum_{i=0}^{|w|-1}\vert w_i\vert_v.$$	
\end{remark}

\begin{example}
Suppose $\Sigma=\{a,b,c\}$ and $[w]=[abcabc]$. By Definition~\ref{DefParikhCircular},  $$|[w]|_{ab}=\frac{1}{3}(|abcabc|_{ab}+|bcabca|_{ab}+|cabcab|_{ab})=\frac{7}{3}.$$ 
\end{example}

\begin{remark}\label{RemSingleLet}
For any $a\in\Sigma$, it holds that $|[w]|_a=|w|_a$ due to the simple fact that $|u|_a=|w|_a$ for every $u\in[w]$.
\end{remark}

In comparison to Proposition~\ref{ProSPandCirc}, the next theorem, which uses the average method of counting, extends Proposition~\ref{CombPropLinWord} to circular words more naturally.

\begin{theorem}\label{CombPropTerWord}
Suppose $\Sigma=\{a_1,a_2,\ldots,a_s\}$ and $[w]\in \displaystyle{\Sigma_c^\ast}$. 
Then 
$$\sum\limits_{\sigma\in\sym(s)}|[w]|_{a_{\sigma(1)}a_{\sigma(2)}\cdots a_{\sigma(s)}} =\prod\limits_{i=1}^{s}|[w]|_{a_i}.$$
\end{theorem}

\begin{proof}
By Definition~\ref{DefParikhCircular} and Proposition~\ref{CombPropLinWord}, we have
\begin{equation*}
\begin{aligned}
\sum\limits_{\sigma\in\sym(s)}|[w]|_{a_{\sigma(1)}a_{\sigma(2)}\cdots a_{\sigma(s)}}&=\sum\limits_{\sigma\in\sym(s)}\frac{1}{|[w]|}\sum\limits_{u\in[w]}|u|_{a_{\sigma(1)}a_{\sigma(2)}\cdots a_{\sigma(s)}}\\
			  &=\frac{1}{|[w]|}\sum\limits_{u\in[w]}\sum\limits_{\sigma\in\sym(s)}|u|_{a_{\sigma(1)}a_{\sigma(2)}\cdots a_{\sigma(s)}}\\			  &=\frac{1}{|[w]|}\sum\limits_{u\in[w]}\prod\limits_{i=1}^{s}|u|_{a_i}\\
			  &=\frac{1}{|[w]|}\sum\limits_{u\in[w]}\prod\limits_{i=1}^{s}|w|_{a_i}\\			  
			  &=\frac{1}{|[w]|}\left( |[w]|\prod\limits_{i=1}^{s}|w|_{a_i}\right)\\
			  &=\prod\limits_{i=1}^{s}|[w]|_{a_i}.
\end{aligned}
\end{equation*}
where the last equality holds by Remark~\ref{RemSingleLet}.
\end{proof}

\section{Parikh Matrices of Circular Words}\label{SecParMat}
We begin this section by presenting the motivation that led to our definition of Parikh matrices for circular words (see Definition~\ref{ParMatCircularWords}).

The study of Parikh matrices in relation to conjugacy classes has been done in \cite{DHMR21}---in particular, on the Parikh matrix of the lexicographically smallest word in a conjugacy class (i.e.~Lyndon conjugate). This Parikh matrix, however, reflects only on the Lyndon conjugate and not on the other words belonging to the conjugacy class. Hence, we aim to propose a definition that takes into account, every word in the conjugacy class. With our proposed definition, two circular words can have distinct Parikh matrices yet their Lyndon conjugates share the same Parikh matrix. (It can be verified using Definition~\ref{ParMatCircularWords} that the circular words $[abcabc]$ and $[abacbc]$ over $\{a<b<c\}$ provide an example of this.)

Subramanian et al., partially motivated by a special sum of Parikh matrices 
introduced by Mateescu \cite{aM04}, have proposed the idea of defining the Parikh matrix of an array by taking the sum of the Parikh matrices of the rows (thus the order of the rows is irrelevant) while keeping the entries of the main diagonal as units \cite{SMAN13}. Since a circular word is essentially a conjugacy class, the ordering of the finite number of words in the class does not matter, and thus it seems that the definition of Parikh matrices for arrays can be adopted to circular words. However, this is not desirable as it is no different from making all the words in the conjugacy class the rows of an array (in any order) and taking its Parikh matrix. Furthermore, the ``Parikh vector" of the circular word (as necklace) is not embedded in the second diagonal of the Parikh matrix this way, contrary to what is stated in Remark~\ref{010721}.

Viewed as a necklace, a circular word is simply one word having a circular structure without a first or last letter. From a fuzzy perspective, we can view the circular word as any word in the conjugacy class with a uniform probability. Therefore, we take an average approach in our definition below and naturally, it is consistent with Definition~\ref{DefParikhCircular}.

\begin{definition}\label{ParMatCircularWords}
	Suppose $\Sigma$ is an ordered alphabet and $[w]\in \displaystyle{\Sigma_c^\ast}$. The \textit{Parikh matrix of the circular word $[w]$ with respect to $\Sigma$}, denoted $\Psi_\Sigma([w])$, is defined by
	$$\Psi_\Sigma([w])=\frac{1}{|[w]|}\sum\limits_{u\in[w]}\Psi_\Sigma(u).$$
\end{definition}

The following property, analogous to that of Parikh matrices of linear words, follows easily by Theorem~\ref{1206a}.
\begin{theorem}\label{TheoEntryParMatCirc}
	Suppose $\Sigma=\{a_1<a_2< \dotsb<a_s\}$ and $[w]\in \displaystyle{\Sigma_c^\ast}$. The matrix $\Psi_{\Sigma}([w])=M$ has the following properties:
	\begin{itemize}
		\item $M_{i,i}=1$ for each $1\leq i \leq s+1$;
		\item $M_{i,j}=0$ for each $1\leq j<i\leq s+1$;
		\item $M_{i,j+1}=\vert [w] \vert_{a_ia_{i+1}\dotsm a_j}$ for each $1\leq i\leq j \leq s$.
	\end{itemize}
\end{theorem}

\begin{example}
	Suppose $\Sigma=\{a<b<c\}$ and  $[w]=[cabacb]$ as in Example~\ref{ExTechnique1}, then: 
	\begin{align*}
		\Psi_{\Sigma}([w])&=\dfrac{1}{6}\left[\Psi_{\Sigma}(cabacb)+\Psi_{\Sigma}(abacbc)+\cdots +\Psi_{\Sigma}(bcabac)\right]\\
		&=\dfrac{1}{6}\left[\begin{pmatrix}
			1 & 2 & 3 & 1\\
			0 & 1 & 2 & 1\\
			0 & 0 & 1 & 2\\
			0 & 0 & 0 & 1
		\end{pmatrix}+
		\begin{pmatrix}
			1 & 2 & 3 & 4\\
			0 & 1 & 2 & 3\\
			0 & 0 & 1 & 2\\
			0 & 0 & 0 & 1
		\end{pmatrix}+
		\dotsm
		+\begin{pmatrix}
			1 & 2 & 1 & 1\\
			0 & 1 & 2 & 3\\
			0 & 0 & 1 & 2\\
			0 & 0 & 0 & 1
		\end{pmatrix}\right]\\
		&=\begin{pmatrix}
			1 & 2 & 2 & \frac{4}{3} \\
			0 & 1 & 2 & 2\\
			0 & 0 & 1 & 2\\
			0 & 0 & 0 & 1
		\end{pmatrix}
		=\begin{pmatrix}
			1 & \vert [w]\vert_a & \vert [w]\vert_{ab} & \vert [w]\vert_{abc} \\
			0 &1 & \vert [w]\vert_b & \vert [w]\vert_{bc}\\
			0 & 0 & 1 & \vert [w]\vert_c\\
			0 & 0 & 0 & 1
		\end{pmatrix}.
	\end{align*}
\end{example}

We now investigate some properties of Parikh matrices of circular words, exclusively for the binary and the ternary alphabets. Our main theorem (Theorem~\ref{BinaryEntries}) shows that for the binary alphabet, $\Psi_\Sigma([w])$ depends only on the Parikh vector of $w$. Before that, we recall an early result regarding inverses of Parikh matrices. 

\begin{definition}
Suppose $\Sigma$ is an ordered alphabet with $\vert \Sigma\vert=s$ and $w\in \displaystyle{\Sigma^\ast}$. Let \mbox{$A=\Psi_\Sigma(w)$}. The \textit{alternate Parikh matrix of $w$}, denoted by $\overline{\Psi}_\Sigma(w)$, is the $(s+1)\times (s+1)$ matrix $B$ such that $B_{i,j}=(-1)^{i+j}A_{i,j}$ for all integers $1\le i,j\le s+1$.
\end{definition}

\begin{theorem}\label{InvAltLinear}\cite{MSSY01}
Suppose $\Sigma$ is an ordered alphabet and $w\in \displaystyle{\Sigma^\ast}$. Then $$[\Psi_\Sigma(w)]^{-1}=\overline{\Psi}_\Sigma(\mi(w)).$$
\end{theorem}

\begin{remark}\label{RemarkInvAlt}
Since the determinant of a Parikh matrix is always one, it follows that the inverse of a Parikh matrix is simply its adjoint. Thus, it can be verified based on Theorem~\ref{1206a} that for $\Sigma=\{a<b<c\}$ and $w\in  \displaystyle{\Sigma^\ast}$, we have 
$[\Psi_\Sigma(w)]^{-1}=$ $$\begin{pmatrix}
	1 & -|w|_a & |w|_a|w|_b-|w|_{ab} & -|w|_a|w|_b|w|_c+|w|_a|w|_{bc}+|w|_{ab}|w|_c-|w|_{abc}\\
	0 &1 & -|w|_b & |w|_b|w|_c-|w|_{bc}\\
	0 & 0 & 1 & -|w|_c\\
	0 & 0 & 0 & 1
\end{pmatrix}.$$
Furthermore, by Theorem~\ref{InvAltLinear}, this matrix is equal to $\overline{\Psi}_\Sigma(\mi(w))$. 
In particular, this implies the identity
$$\vert \mi(w)\vert_{abc} =|w|_a|w|_b|w|_c-|w|_a|w|_{bc}-|w|_{ab}|w|_c+|w|_{abc}.$$
This observation has also been highlighted in \cite{APT19} and \cite{MSY04}.
\end{remark}

\begin{proposition}\label{PropInvAndAlt}
Suppose $\Sigma$ is an ordered alphabet with $|\Sigma|\le 3$ and $w\in \displaystyle{\Sigma^\ast}$. Then $$[\Psi_\Sigma([w])]^{-1}=\overline{\Psi}_\Sigma([\mi(w)]).$$
\end{proposition}
\begin{proof}
We show only the proof for the ternary alphabet as the case for the binary alphabet follows by similar argument. Let $\Sigma=\{a<b<c\}$.

As in Remark~\ref{RemarkInvAlt}, since the inverse of a Parikh matrix $\Psi_{\Sigma}([w])$ is its adjoint,  it can be verified based on Theorem~\ref{TheoEntryParMatCirc} that the $(1,4)$-entry of $[\Psi_\Sigma([w])]^{-1}$ is $-(|[w]|_a|[w]|_b|[w]|_c-|[w]|_a|[w]|_{bc}-|[w]|_{ab}|[w]|_c+|[w]|_{abc})$. On the other hand, the $(1,4)$-entry of $\overline{\Psi}_\Sigma([\mi(w)])$ is $-|[\mi(w)]|_{abc}$. Thus it remains to see that
\begin{equation*}
	\begin{aligned}
		|[\mi(w)]|_{abc}&=\dfrac{1}{|[\mi(w)]|}\sum\limits_{u\in [\mi(w)]}|u|_{abc}\\
		&=\dfrac{1}{|[w]|}\sum\limits_{u\in [w]}|\mi(u)|_{abc}\\
		&=\dfrac{1}{|[w]|}\sum\limits_{u\in [w]}(|u|_a|u|_b|u|_c-|u|_a|u|_{bc}-|u|_{ab}|u|_c+|u|_{abc}) \text{ (By Remark~\ref{RemarkInvAlt})}\\
		&=\dfrac{1}{|[w]|}(|w|_a|w|_b|w|_c |[w]|-|w|_a\sum\limits_{u\in [w]}|u|_{bc}-|w|_c\sum\limits_{u\in [w]}|u|_{ab}+\sum\limits_{u\in [w]}|u|_{abc})\\
		&=|[w]|_a|[w]|_b|[w]|_c-|[w]|_a|[w]|_{bc}-|[w]|_{ab}|[w]|_c+|[w]|_{abc}.
		\end{aligned}
\end{equation*}
Likewise, it can be shown that the two matrices agree at the other entries and thus they are equal.		
\end{proof}

\begin{remark}\label{RemarkPower}
For any $A\in \mathcal{M}_4$ and positive integer $p$, by Theorem~3.1 in \cite{GT17d}, it holds that $ \displaystyle{A^p}\in \mathcal{M}_4$ and
\begin{itemize}
	\item $A^p_{1,2}= pA_{1,2}$, $A^p_{2,3}= pA_{2,3}$, $A^p_{3,4}= pA_{3,4}$,\\
	\item $A^p_{1,3}= pA_{1,3}+\binom{p}{2}A_{1,2}A_{2,3}$,\\
	\item $A^p_{2,4}= pA_{2,4}+\binom{p}{2}A_{2,3}A_{3,4}$,\\
	\item $A^p_{1,4}= pA_{1,4}+\binom{p}{2}(A_{1,2}A_{2,4}+A_{1,3}A_{3,4})+ \binom{p}{3}A_{1,2}A_{2,3}A_{3,4}$.
\end{itemize}
A similar result holds for any $A\in \mathcal{M}_3$.
\end{remark}

\begin{proposition}\label{PropPower}
Suppose $\Sigma$ is an ordered alphabet with $|\Sigma|\le 3$ and $w\in \displaystyle{\Sigma^\ast}$. For any positive integer $p$, $\Psi_{\Sigma}([w^p])= \left[\Psi_{\Sigma}([w])  \right]^p$.
\end{proposition}

\begin{proof}
We show only the proof for the case $|\Sigma|=3$ as the other case follows by similar argument. Let $\Sigma=\{a<b<c\}$ and fix an arbitrary positive integer $p$.
By Theorem~\ref{TheoEntryParMatCirc} and Remark~\ref{RemarkPower}, since 
$\Psi_{\Sigma}([w])\in \mathcal{M}_4$,
 it follows that the $(1,4)$-entry of $\left[ \Psi_{\Sigma}([w])  \right]^p$ is
$$p | [w]|_{abc}+ \binom{p}{2} ( |[w]|_a |[w]|_{bc}+  |[w]|_{ab}|[w]|_{c})+\binom{p}{3}|[w]|_a|[w]|_{b}|[w]|_c.$$ 
On the other hand, the $(1,4)$-entry of $\Psi_{\Sigma}([w^p])$ is
\begin{equation*}
\begin{aligned}
|[w^p]|_{abc}
&=\dfrac{1}{|[w^p]|}\sum\limits_{u\in [w^p]}|u|_{abc}\\
&=\dfrac{1}{|[w]|}\sum\limits_{v\in [w]}|v^p|_{abc}\\
&=\dfrac{1}{|[w]|}\sum\limits_{v\in [w]}\left(p|v|_{abc}+\binom{p}{2}|v|_a|v|_{bc}
+\binom{p}{2}|v|_{ab}|v|_{c} + \binom{p}{3}|v|_a|v|_{b}|v|_c \right)\\
&=\dfrac{1}{|[w]|}
\left[p \sum\limits_{v\in [w]}|v|_{abc}+ \binom{p}{2} |w|_a \sum\limits_{v\in [w]}|v|_{bc} +  \binom{p}{2} |w|_c \sum\limits_{v\in [w]}|v|_{ab}+  \binom{p}{3}|w|_a|w|_{b}|w|_c | [w]|  \right]
\\
&=p | [w]|_{abc}+ \binom{p}{2} |[w]|_a |[w]|_{bc}+  \binom{p}{2} |[w]|_{ab} |[w]|_{c}+\binom{p}{3}|[w]|_a|[w]|_{b}|[w]|_c,  
\end{aligned}
\end{equation*}	
where the third equality follows by Remark~\ref{RemarkPower} as $\Psi_{\Sigma}(v)\in \mathcal{M}_4$ and 
$\Psi_{\Sigma}(v^p)= \left[\Psi_{\Sigma}(v)  \right]^p$ since the Parikh matrix mapping of linear words is a morphism.
Hence, $\Psi_{\Sigma}([w^p])$ and $\left[ \Psi_{\Sigma}([w])\right]^p$ agree at the $(1,4)$-entry. Likewise, it can be shown that the two matrices agree at the other entries and thus they are equal.		
\end{proof}

The following simple example shows that Proposition~\ref{PropInvAndAlt} and Proposition~\ref{PropPower} cannot be extended to cater for larger alphabets.
\begin{example}
Consider the circular word $[w]=[abcd]$ over the ordered alphabet \mbox{$\{a<b<c<d\}$.} The matrices $A=[\Psi_\Sigma([w])]^{-1}$ and $B=\overline{\Psi}_\Sigma([\mi(w)])$  agree everywhere except at the top right entry, where $A_{1,5}=\frac{1}{16}\neq 0=B_{1,5}$.
Similarly, the matrices $C=\Psi_{\Sigma}([w^2])$ and $D=\left[\Psi_{\Sigma}([w])\right]^2$ 
agree everywhere except at the top right entry, where $C_{1,5}=2\neq\frac{33}{16}=D_{1,5}$.
\end{example}

\begin{theorem}\label{BinaryEntries}
Suppose $\Sigma=\{a<b\}$ and $w\in \displaystyle{\Sigma^\ast}$. Then
\begin{equation*}
\Psi_\Sigma([w])=
\begin{pmatrix}
1 & |w|_a & \dfrac{|w|_a|w|_b}{2}\\
0 & 1     & |w|_b\\
0 & 0     & 1
\end{pmatrix}.
\end{equation*}
\end{theorem}
\begin{proof}
First, we assume that $w$ is primitive.
Apart from the top right entry, the other entries easily follow from Theorem~\ref{TheoEntryParMatCirc} and Remark~\ref{RemSingleLet}. Thus it remains to show that $$|[w]|_{ab}=\dfrac{1}{|[w]|}\sum\limits_{u\in [w]}|u|_{ab}=\dfrac{|w|_a|w|_b}{2}.$$ 

If $|w|_a=0$ or $|w|_b=0$, then the conclusion holds. Assume $|w|_a \ge 1$ and $|w|_b\ge 1$. 

Let $i_1<i_2<\ldots <i_{|w|_a}$ be integers such that $w[i_k]=a$ for all integers $1\le k\le |w|_a$. Similarly, let $j_1<j_2<\ldots <j_{|w|_b}$ be integers such that $w[j_k]=b$ for all integers $1\le k\le |w|_b$. Let $I=\{i_1,i_2,\ldots ,i_{|w|_a}\}$ and $J=\{j_1,j_2,\ldots ,j_{|w|_b}\}$. 

For every pair of integers $i\in I$ and $j\in J$, define $d(i,j)$ to be the number of words $u\in[w]$ such that the pair of letters $w[i]$ and $w[j]$ is counted towards the value of $|u|_{ab}$. 
Note that $\sum_{u\in [w]}|u|_{ab}=\sum_{h=1}^{|w|_a}\sum_{k=1}^{|w|_b}d(i_h,j_k)$. Furthermore, by some simple observation and due to our assumption that $w$ is primitive, we have 
$d(i,j)=\begin{cases}
|w|+i-j,	&\text{ if } i<j\\
i-j,		&\text{ if } i>j.
\end{cases}$

Fix an arbitrary integer $1\le h\le |w|_a$. Then $w=xw[i_h]y$ for some $x,y\in \displaystyle{\Sigma^\ast}$. We have $|xw[i_h]|_b=|xw[i_h]|-|xw[i_h]|_a=i_h-h$. Thus for every integer $1\le k\le i_h-h$, we have $i_h>j_k$. On the other hand, for every integer $i_h-h+1\le k\le |w|_b$, we have $i_h<j_k$. Following this, it holds that
\begin{equation*}
\begin{aligned}
\sum_{k=1}^{|w|_b}d(i_h,j_k)&=\sum_{\substack{1\le k\le |w|_b\\i_h>j_k}}d(i_h,j_k)+\sum_{\substack{1\le k\le |w|_b\\i_h<j_k}}d(i_h,j_k)\\
							&=\sum_{k=1}^{i_h-h}(i_h-j_k)+\sum_{k=i_h-h+1}^{|w|_b}(|w|+i_h-j_k)\\
							&=|w|(|w|_b-i_h+h)+|w|_bi_h-\sum_{k=1}^{|w|_b}j_k\\
							&=|w|_b|w|+|w|h-|w|_ai_h-\sum_{k=1}^{|w|_b}j_k. 
\end{aligned}
\end{equation*}
It remains to see that  
\begin{equation*}
\begin{aligned}
\sum_{u\in [w]}|u|_{ab}&=\sum_{h=1}^{|w|_a}\sum_{k=1}^{|w|_b}d(i_h,j_k)\\
                       &=|w||w|_a|w|_b+|w|\sum_{h=1}^{|w|_a}h-|w|_a\left(\sum_{h=1}^{|w|_a}i_h+\sum_{k=1}^{|w|_b}j_k\right)\\
                       &=|w||w|_a|w|_b+|w|\dfrac{|w|_a(|w|_a+1)}{2}-|w|_a\dfrac{|w|(|w|+1)}{2}\\
                       &=\dfrac{|w||w|_a}{2}(2|w|_b+|w|_a+1-|w|-1)\\
                       &=\dfrac{|w||w|_a|w|_b}{2}.
\end{aligned}
\end{equation*}
That is to say, $| [w]|_{ab}=\dfrac{1}{|w|}\sum\limits_{u\in[w]}|u|_{ab}=\dfrac{|w|_a|w|_b}{2}$ as required. 

Now, assume $w$ is not primitive. Then $w=v^p$ for some primitive word $v\in \displaystyle{\Sigma^\ast}$ and integer $p\ge 2$. 
By Proposition~\ref{PropPower},
$\Psi_{\Sigma}([w])=\Psi_{\Sigma}([v^p])=   [\Psi_{\Sigma}([v])]^p$.
Since $v$ is primitive, by what we have shown above, $\Psi_{\Sigma}([v])
= \begin{psmallmatrix}
1 & |v|_a & \frac{|v|_a|v|_b}{2}\\
0 & 1     & |v|_b\\
0 & 0     & 1
\end{psmallmatrix}.$
By Remark~\ref{RemarkPower},
$$
\Psi_{\Sigma}([w])
= \left[\begin{pmatrix}
1 & |v|_a & \dfrac{|v|_a|v|_b}{2}\\
0 & 1     & |v|_b\\
0 & 0     & 1
\end{pmatrix}\right]^p
= \begin{pmatrix}
1 & p|v|_a & \dfrac{p|v|_a|v|_b}{2}+\binom{p}{2}|v|_a|v|_b\\
0 & 1     & p|v|_b\\
0 & 0     & 1
\end{pmatrix}.$$
It follows that $\Psi_{\Sigma}([w])$ has the required form as $$\dfrac{p|v|_a|v|_b}{2}+\dfrac{p(p-1)}{2}|v|_a|v|_b=\dfrac{p|v|_a\cdot p|v|_b}{2}=\dfrac{|w|_a|w|_b}{2}.$$

Thus the conclusion holds.
\end{proof}

\begin{remark}
 The identity $|[w]|_{ab}=
 \frac{|w|_a|w|_b}{2}$ presented in Theorem~\ref{BinaryEntries} does not hold in general for circular words over larger alphabets. This is because the number and the position of other letters, apart from $a$ and $b$, in a circular word affect the count of subword $ab$ in it. In fact, 
 $|[w]|_{ab} $ need not equal $|[\pi_{\{a,b\}}(w)]|_{ab}$ (see Definition~\ref{DefProjection}); for example,
 $|[abc]|_{ab}= \frac{2}{3}$ but
 $|[ab]|_{ab}=1$.
\end{remark}

By Theorem~\ref{BinaryEntries}, one can see that the Parikh matrix of a binary circular word $[w]$ over the ordered alphabet $\{a<b\}$ depends only on $|w|_a$, since the number of occurrences of the other letter is simply $|w|-|w|_a$. 
A simple observation that follows is that the number of distinct Parikh matrices corresponding to binary circular words of length $n$ is $n+1$.

The original Parikh matrix mapping of linear words is a morphism by default of its definition. However, the Parikh matrix mapping of circular words is not a morphism. In fact, the equality $\Psi_{\Sigma}([u] [v])=\Psi_{\Sigma}([u])\Psi_{\Sigma}([v])$ does not make sense as there is no canonical concatenation of two circular words. However, it is natural to ask whether $\Psi_{\Sigma}([uv])=\Psi_{\Sigma}([u])\Psi_{\Sigma}([v])$ holds generally. Our next corollary provides a negative answer to this.  

\begin{corollary}\label{ConditionBinary}
Suppose $\Sigma=\{a<b\}$ and $u,v\in \displaystyle{\Sigma^\ast}$. We have $$\Psi_{\Sigma}([uv])=\Psi_{\Sigma}([u])\Psi_{\Sigma}([v])$$ if and only if $|u|_a|v|_b=|v|_a|u|_b$.
\end{corollary}
\begin{proof} 
By Theorem~\ref{BinaryEntries}, we have 
$$\Psi_{\Sigma}([uv])=\begin{pmatrix}
1 & |uv|_a & \dfrac{|uv|_a|uv|_b}{2}\\
0 & 1     & |uv|_b\\
0 & 0     & 1
\end{pmatrix}$$ and
$$\Psi_{\Sigma}([u])\Psi_{\Sigma}([v])=\begin{pmatrix}
1 & |u|_a +|v|_a & \dfrac{|u|_a|u|_b+2|u|_a|v|_b+|v|_a|v|_b}{2}\\
0 & 1     & |u|_b+ |v|_b\\
0 & 0     & 1
\end{pmatrix}.$$
It remains to note that the equality of the top right entries of both matrices reduces to $|u|_a|v|_b=|v|_a|u|_b$. Thus our conclusion holds.
\end{proof}

When $|u|_a|v|_b=|v|_a|u|_b$, the words $u$ and $v$ are said to satisfy a \emph{weak ratio property} \cite{MS12}. It turns out that two Parikh matrices of binary circular words commute if and only if the representatives of the circular words satisfy a weak ratio property as well. The similar proof is omitted.

\begin{corollary}
Suppose $\Sigma=\{a<b\}$ and $u,v\in \displaystyle{\Sigma^\ast}$. We have $$\Psi_{\Sigma}([u])\Psi_{\Sigma}([v])=\Psi_{\Sigma}([v])\Psi_{\Sigma}([u])$$ if and only if  $|u|_a|v|_b=|v|_a|u|_b$.
\end{corollary}

\section{Matrix-equivalence of Circular Words}

As in the case of linear words, the Parikh matrix defined in Definition~\ref{DefParikhCircular} does not necessarily characterize a circular word uniquely. For example, with respect to $\Sigma=\{a<b\}$, we have $\Psi_{\Sigma}([abab])=\begin{pmatrix}
1 &2 &2\\
0 &1 &2\\
0 &0 &1
\end{pmatrix}=\Psi_{\Sigma}([bbaa])$. Thus, we define the following notion, using the terminology for the case of linear words.

\begin{definition}
Suppose $\Sigma$ is an ordered alphabet and $w,w'\in \displaystyle{\Sigma^\ast}$. The circular words $[w]$ and  $[w']$ \textit{are matrix equivalent (or simply $M$\!-equivalent)}, denoted by $[w]\equiv_M[w']$, if and only if $\Psi_{\Sigma}([w])=\Psi_{\Sigma}([w'])$.
\end{definition}

The following question then naturally follows. 
\begin{question}
Suppose $\Sigma$ is an ordered alphabet and  $w,w'\in \displaystyle{\Sigma^\ast}$  with $w \not\sim_c w'$. When are $[w]$ and $[w']$ $M$\!-equivalent? 
\end{question}

For the binary alphabet, the answer to the above question is simple due to Theorem~\ref{BinaryEntries}---two distinct circular words $[w]$ and $[w']$ share the same Parikh matrix if and only if $\Psi(w)=\Psi(w')$ (i.e. they have the same Parikh vector). However, we haven't been able to obtain a complete characterization for the ternary alphabet. 

In an attempt to characterize linear words having the same Parikh matrix, two elementary rewriting rules were introduced in \cite{AAP08}---these rules and their generalization have then been further studied in the literature. The following version is formulated exclusively for the ternary alphabet. Suppose $\Sigma=\{a<b<c\}$ and $w,w'\in \displaystyle{\Sigma^\ast}$.

\vspace{0.4em}\begin{itemize}[leftmargin=1.78cm]
\item[Rule $E_1$.] If $w=xacy$ and $w'=xcay$ for some $x,y\in \displaystyle{\Sigma^\ast}$, then $\Psi_{\Sigma}(w)=\Psi_{\Sigma}(w')$.
\item[Rule $E_2$.] If $w=x\alpha byb\alpha z$ and $w'=xb\alpha y\alpha bz$ for some $\alpha\in\{a,c\}$, $x,z\in \displaystyle{\Sigma^\ast}$ and $y\in \displaystyle{ \{\alpha,b\}^\ast }$, then $\Psi_{\Sigma}(w)=\Psi_{\Sigma}(w')$.
\end{itemize}

The above rules, however, when applied on circular words, do not preserve \mbox{$M$\!-equivalence} in general. For example,
\begin{center}
\begin{tikzpicture}[>=stealth',semithick,auto,]
[scale=.8,auto=left,every node/.style={}]
\node at (0,0) {$\boldsymbol{a}$};
\node at (-0.35,-0.55) {$b$};
\node at (0.35,-0.6) {$\boldsymbol{c}$};

\node (c1) at (0.7,-0.3) {};
\node (c2) at (2.3,-0.3) {};

\node at (3,0) {$\boldsymbol{c}$};
\node at (2.60,-0.55) {$b$};
\node at (3.35,-0.6) {$\boldsymbol{a}$};

\path[->](c1) edge[] node {{\tiny Rule $E_1$}} (c2);
\end{tikzpicture}
\end{center}
but $|[acb]|_{ab}=\dfrac{1}{3}\neq\dfrac{2}{3}=|[cab]|_{ab}$. A simple counterexample for Rule $E_2$ is the pair of circular words $[abbac]$ and $[baabc]$ as $|[abbac]|_{abc}=\dfrac{2}{5}\neq 1=|[baabc]|_{abc}$.

We now develop some natural $M$\!-equivalence preserving rewriting rules for ternary circular words, analogous to Rule $E_1$ and Rule $E_2$. For this purpose, we need the following technical lemma.  Recall that $\SP_{\Sigma}$ is the set of slender Parikh words over $\Sigma$, as defined in Definition~\ref{010921d}.

\begin{lemma}\label{LemRuleCE2}
 Let $\Sigma=\{a,b,c\}$ and suppose that $w=x\alpha\beta y\beta\alpha$ and $w'=x\beta\alpha y\alpha\beta$ for some distinct $\alpha,\beta\in\Sigma$ and $x,y\in \displaystyle{\Sigma^\ast}$. Then, $$|[w]|_u=|[w']|_u \text{ for all } u\in \SP_{\Sigma} \text{ with } |u|\leq 2.$$
\end{lemma}

\begin{proof}
	Clearly, $|[w]|_u=|w|_u=|w'|_u=|[w']|_u$ for all $u\in\{a,b,c\}$. 
	Hence, suppose $u\in \SP_{\Sigma}$ and $|u|=2$.	
	Let $n=|w|=|w'|$. By Remark~\ref{AltDef}, it suffices to show that
	$\sum_{i=0}^{n-1}\vert w_i\vert_u=\sum_{i=0}^{n-1} \vert w'_i\vert_u$, where $w_i$ and $w'_i$ are the $i$-th cyclic shifts of $w$ and $w'$ respectively. Since $\alpha$ and $\beta$ are distinct and we can interchange $w$ and $w'$ if necessary, it suffices to consider the cases where $u=\alpha\beta$ and $u=\beta\gamma$, where $\gamma\in \Sigma\backslash \{\alpha, \beta\}$. 
Before that, note that $w_{\vert x\vert+1}= \beta y\beta\alpha x \alpha$, $w'_{\vert x\vert+1}= \alpha y\alpha\beta x \beta$, $w_{n-1}=\alpha x\alpha\beta y\beta$, and 
$w'_{n-1}=\beta x\beta\alpha y\alpha$.

\begin{case}$u=\alpha\beta$.

Notice that for every integer $i\in\{0,1,\ldots ,n-1\}\backslash\{|x|+1,n-1\}$, the word $\pi_{\{\alpha,\beta\}}(w_i)$ can be rewritten into $\pi_{\{\alpha,\beta\}}(w'_i)$ by an application of Rule $E_2$, thus $|w_i|_u=|\pi_{\{\alpha,\beta\}}(w_i)|_{\alpha\beta}=|\pi_{\{\alpha,\beta\}}(w'_i)|_{\alpha\beta}=|w'_i|_u$. In the other cases, we have
	\begin{flalign*}
		|w_{|x|+1}|_{u}=|\beta y\beta\alpha x \alpha|_{\alpha\beta}&=|y|_{\alpha\beta}+|y|_\alpha(1+|x|_\beta)+|x|_\beta+|x|_{\alpha\beta},\\
		|w_{n-1}|_{u}=|\alpha x\alpha\beta y\beta|_{\alpha\beta}&=|x|_\beta+|y|_\beta+2+|x|_{\alpha\beta}+|x|_\alpha(2+|y|_\beta)+|y|_\beta+2+|y|_{\alpha\beta}+|y|_\alpha,\\
		|w'_{|x|+1}|_{u}=|\alpha y\alpha\beta x \beta|_{\alpha\beta}&=|y|_\beta+|x|_\beta+2+|y|_{\alpha\beta}+|y|_\alpha(2+|x|_\beta)+ |x|_{\beta}+2+ |x|_{\alpha\beta}+|x|_\alpha,\\
	|w'_{n-1}|_{u}=|\beta x\beta\alpha y\alpha|_{\alpha\beta}&=|x|_{\alpha\beta}+|x|_\alpha (1+|y|_\beta)+|y|_\beta+|y|_{\alpha\beta}.
	&&
	\end{flalign*}
\end{case}
	
\begin{case}$u=\beta\gamma$, where $\gamma\in \Sigma\backslash \{\alpha, \beta\}$.
	
	Notice that for every integer $i\in\{0,1,\ldots ,n-1\}\backslash\{|x|+1,n-1\}$, we have $\pi_{\{\beta, \gamma\}}(w_i)=\pi_{\{\beta, \gamma\}}(w'_i)$, thus $|w_i|_u=|\pi_{\{\beta,\gamma\}}(w_i)|_{\beta\gamma}=|\pi_{\{\beta, \gamma\}}(w'_i)|_{\beta\gamma}=|w'_i|_u$. Meanwhile, we have
	\begin{flalign*}
	|w_{|x|+1}|_{u}=|\beta y\beta\alpha x\alpha|_{\beta\gamma}&=|y|_\gamma+|x|_\gamma+|y|_{\beta\gamma}+|y|_\beta|x|_\gamma+|x|_\gamma+|x|_{\beta\gamma},\\
	|w_{n-1}|_{u}=  |\alpha x \alpha \beta y\beta|_{\beta\gamma}&= |x|_{\beta\gamma}+ |x|_\beta|y|_\gamma+|y|_\gamma+|y|_{\beta\gamma},\\
	|w'_{|x|+1}|_{u}=|\alpha y\alpha \beta x\beta|_{\beta\gamma}&=|y|_{\beta\gamma}+ |y|_\beta|x|_\gamma+|x|_\gamma+|x|_{\beta\gamma},\\
	|w'_{n-1}|_{u}=|\beta x\beta\alpha y\alpha|_{\beta\gamma}&=|x|_\gamma+|y|_\gamma+ |x|_{\beta\gamma}+ |x|_\beta|y|_\gamma+|y|_\gamma+|y|_{\beta\gamma}.
	&&
\end{flalign*}
\end{case}	
In either case, it can be verified that $|w_{|x|+1}|_{u}+|w_{n-1}|_{u}=|w'_{|x|+1}|_{u}+|w'_{n-1}|_{u}$ and thus
$\sum_{i=0}^{n-1}\vert w_i\vert_u=\sum_{i=0}^{n-1} \vert w'_i\vert_u$ as required.	
\end{proof}

\begin{theorem}\label{TheoRuleCE1}
 Let $\Sigma=\{a<b<c\}$. Suppose that $w=xacyca$ and $w'=xcayac$  for some $x,y\in \displaystyle{\Sigma^\ast}$. We have $[w]\equiv_M[w']$ if and only if $|y|_b(|x|_a-|x|_c) = |x|_b(|y|_a-|y|_c)$.
\end{theorem}
\begin{proof}
 By Lemma~\ref{LemRuleCE2},  
$|[w]|_u=|[w']|_u \text{ for all } u\in \{a,b,c,ab,bc  \}$.	Let $n=|w|=|w'|$. By Theorem~\ref{TheoEntryParMatCirc} and Remark~\ref{AltDef}, it suffices to show that $\sum_{i=0}^{n-1} \vert w_i\vert_{abc}=\sum_{i=0}^{n-1} \vert w'_i\vert_{abc}$  if and only if $|y|_b(|x|_a-|x|_c) = |x|_b(|y|_a-|y|_c)$, where $w_i$ and $w'_i$ are the $i$-th cyclic shifts of $w$ and $w'$ respectively.

For every integer $i\in\{0,1,\ldots ,n-1\}\backslash$ $\{|x|+1,n-1\}$, observe that $w'_i$ can be obtained from $w'_i$ by two applications of Rule~$E_1$---thus by the transitivity of $M$\!-equivalence, it follows that $w_i \equiv_M w_i'$ and thus
$|w_i|_{abc}=|w'_i|_{abc}$. Meanwhile, 
	\begin{flalign*}
		|w_{|x|+1}|_{abc}=|cycaxa|_{abc}={}&|y|_{abc}+ |y|_{ab}(1+|x|_c)+ |y|_a|x|_{bc}+ |x|_{bc} +|x|_{abc},\\		
		|w'_{|x|+1}|_{abc}=|ayacxc|_{abc}={}&|y|_{bc}+|y|_b(2+|x|_c)+|x|_{bc}+|x|_b+ |y|_{abc}+|y|_{ab}(2+|x|_c) \\
		&+ |y|_a(|x|_{bc}+|x|_b) +|x|_{bc}+|x|_b+|x|_{abc} +|x|_{ab}.
	\end{flalign*}

	Similarly,
	\begin{flalign*}
		|w_{n-1}|_{abc}=|axacyc|_{abc}={}& |x|_{bc}+ |x|_b(2+|y|_c)+|y|_{bc}+|y|_b +|x|_{abc} +|x|_{ab}(2+|y|_c)\\
		& +|x|_a(|y|_{bc}+|y|_{b}) +|y|_{bc} +|y|_b +|y|_{abc}+|y|_{ab}  \\				
		|w'_{n-1}|_{abc}=|cxcaya|_{abc}={}&|x|_{abc}+ |x|_{ab}(1+|y|_c)+|x|_a|y|_{bc}+|y|_{bc} +|y|_{abc}.
	\end{flalign*}

	From the above values, it can be verified carefully that 
	\begin{flalign*}
		\sum_{i=0}^{n-1}|w_i|_{abc}-\sum_{i=0}^{n-1} |w'_i|_{abc}
			&=|w_{|x|+1}|_{abc}- |w'_{|x|+1}|_{abc}+ |w_{n-1}|_{abc}- |w'_{n-1}|_{abc}\\
		&=|x|_a|y|_b+|x|_b|y|_c - |y|_a|x|_b-|y|_b|x|_c.
	\end{flalign*}
Hence, $\sum_{i=0}^{n-1}|w_i|_{abc}=\sum_{i=0}^{n-1} |w'_i|_{abc}$ if and only if $ |y|_b(|x|_a-|x|_c)= |x|_b(|y|_a-|y|_c)$.
\end{proof}

Theorem~\ref{TheoRuleCE1} is the basis of our analogue of Rule $E_1$ for circular ternary words.
In order to prove the corresponding theorem for Rule $E_2$, we need another simple technical lemma, on top of Lemma~\ref{LemRuleCE2}.

\begin{lemma}\cite[Lemma 4]{aS10}\label{LemCounter}
 Let $\Sigma=\{a,b,c\}$ and suppose that $w=x\alpha byb\alpha z$ and $w'=xb\alpha y\alpha bz$  for some $\alpha\in\{a,c\}$ and $x,y,z\in \displaystyle{\Sigma^\ast}$. Let $\overline{\alpha}\in\Sigma\backslash\{\alpha, b\}$. Then, $|w|_{abc}-|w'|_{abc}=|y|_{\overline{\alpha} }$.
\end{lemma}

\begin{theorem}\label{TheoBackBoneCE2}
 Let $\Sigma=\{a<b<c\}$ and suppose that $w=x\alpha byb\alpha$ and $w'=xb\alpha y\alpha b$  for some $\alpha\in\{a,c\}$ and $x,y\in \displaystyle{\Sigma^\ast}$. 
	We have $[w]\equiv_M[w']$ if and only if $|x|_{\overline{\alpha}}(|y|+|y|_b+3)=|y|_{\overline{\alpha}}(|x|+|x|_b+3)$, where $\overline{\alpha}\in\Sigma\backslash\{\alpha, b\}$.
\end{theorem}
\begin{proof}
	Without loss of generality, assume $\alpha=a$ as the other case is similar. Hence, $w=xabyba$ and $w'=xba ya b$.
	 By Lemma~\ref{LemRuleCE2} and Remark~\ref{AltDef}, it suffices to show that $\sum_{i=0}^{n-1} \vert w_i\vert_{abc}=\sum_{i=0}^{n-1} \vert w'\vert_{abc}$  if and only if $|x|_{c}(|y|+|y|_b+3)=|y|_{c}(|x|+|x|_b+3)$, where $n=|w|=|w'|$ and $w_i$ and $w'_i$ are the $i$-th cyclic shifts of $w$ and $w'$, respectively.
	 
	 For every integer $0\le i\le |x|$, by Lemma~\ref{LemCounter}, we have $|w_i|_{abc}-|w'_i|_{abc}=|y|_c$. Similarly, for every integer $|x|+2\le i\le n-2$, by Lemma~\ref{LemCounter}, we have $|w_i|_{abc}-|w'_i|_{abc}=-|x|_c$. Meanwhile, we have
		\begin{flalign*}
		|w_{|x|+1}|_{abc}=|bybaxa|_{abc}={}&|y|_{abc}+|y|_{ab}|x|_c+|y|_a(|x|_c+|x|_{bc})+ |x|_{bc}+|x|_{abc},\\		
		|w'_{|x|+1}|_{abc}=|ayabxb|_{abc}={}&|y|_{bc}+ |y|_b|x|_c+|x|_c+|x|_{bc}+|y|_{abc}+|y|_{ab}|x|_c+|y|_a(|x|_c+ |x|_{bc}) \\
		& +|x|_c+|x|_{bc}+|x|_{abc}.
	\end{flalign*}
	thus $|w_{|x|+1}|_{abc}-|w'_{|x|+1}|_{abc}=-|x|_{bc}-|y|_{bc}-2|x|_c-|y|_b|x|_c$. 
	
	Similarly,
	\begin{flalign*}
		|w_{n-1}|_{abc}=|axabyb|_{abc}&=|x|_{bc}+|x|_b|y|_c+|y|_c+|y|_{bc}+|x|_{abc} +|x|_{ab}|y|_c  +  |x|_a(|y|_c+|y|_{bc})\\
		&\,\,\,\,\,\,+|y|_c+|y|_{bc}+|y|_{abc},\\  		
		|w'_{n-1}|_{abc}=|bxbaya|_{abc}&=|x|_{abc} +|x|_{ab}|y|_c +|x|_a(|y|_c+|y|_{bc})+|y|_{bc}+|y|_{abc}.
	\end{flalign*}
	thus $|w_{n-1}|_{abc}-|w'_{n-1}|_{abc}=|x|_{bc}+|y|_{bc}+2|y|_c+|x|_b|y|_c$.
	
	From the values above, it can be verified that 
	\begin{flalign*}
			\sum_{i=0}^{n-1}|w_i|_{abc}-\sum_{i=0}^{n-1} |w'_i|_{abc}
			={}&\sum_{i=0}^{n-1} \left(|w_i|_{abc}-|w'_i|_{abc}\right)\\
		   	 ={}&(|x|+1)|y|_c + ( -|x|_{bc}-|y|_{bc}-2|x|_c-|y|_b|x|_c       )\\
		   	 &+ (n-|x|-3 )(-|x|_c)+   ( |x|_{bc}+|y|_{bc}+2|y|_c+|x|_b|y|_c   )\\
		   	     ={}&|x||y|_c+  3|y|_c+|x|_b|y|_c  -|y||x|_c-3|x|_c-|y|_b|x|_c,
	\end{flalign*}
where the last equality follows because $n= |x|+|y|+4$.
Therefore, $\sum_{i=0}^{n-1}|w_i|_{abc}=\sum_{i=0}^{n-1} |w'_i|_{abc}$ if and only if $|x|_c(|y|+|y|_b+3)=|y|_c(|x|+|x|_b+3)$.
\end{proof}

By Theorems~\ref{TheoRuleCE1} and \ref{TheoBackBoneCE2}, we state the following
$M$\!-equivalence preserving elementary rules for circular words, exclusively for the ternary alphabet. Suppose 
$\Sigma=\{a<b<c\}$ and $w,w'\in \displaystyle{\Sigma^\ast}$.

\vspace{0.4em}\begin{itemize}[leftmargin=2.10cm]
	\item[Rule $CE_1$.] If $w=xacyca$ and $w'=xcayac$ for some $x,y\in \displaystyle{\Sigma^\ast}$ such that $|y|_b(|x|_a-|x|_c) = |x|_b(|y|_a-|y|_c)$, then $[w]\equiv_M[w']$.
	\item[Rule $CE_2$.] If $w=x\alpha byb\alpha $ and $w'=xb\alpha y\alpha b$ for some $\alpha\in\{a,c\}$ and  $x,y\in \displaystyle{\Sigma^\ast}$  such that $|x|_{\overline{\alpha}}(|y|+|y|_b+3)=|y|_{\overline{\alpha}}(|x|+|x|_b+3)$  where $\overline{\alpha}\in\Sigma\backslash\{\alpha, b\}$, then $[w]\equiv_M[w']$.
\end{itemize}

If $w=xac$ and $w'=xca$ for some $x\in \displaystyle{\Sigma^\ast}$ with $\vert x\vert_b\neq 0$, then $[w]\not\equiv_M [w']$. Meanwhile, if $w=xacyca$ and $w'=xcayac$ for some $x,y\in \displaystyle{\Sigma^\ast}$, then as linear words, it holds that 
  $\Psi_{\Sigma} (w) = \Psi_{\Sigma} (xcayca)= \Psi_{\Sigma}(w')$  by two applications of Rule $E_1$. This brings us to Rule $CE_1$ as the analogue of Rule $E_1$.

The conditions on $x$ and $y$ in Rule $CE_1$ and Rule $CE_2$ are not easy to remember. The following corollary offers a simpler criterion  guaranteeing that the rewriting of circular ternary words as in the rules preserves $M$\!-equivalence.

\begin{corollary}
 Let $\Sigma=\{a,b,c\}$ and suppose that $w=x\alpha\beta y\beta\alpha$ and $w'=x\beta\alpha y\alpha\beta$  for some distinct $\alpha,\beta\in\Sigma$ and $x,y\in \displaystyle{\Sigma^\ast}$. If $\Psi(x)=\Psi(y)$, then $[w]\equiv_M[w']$.
\end{corollary}

\section{Conclusion}

We have presented two different ways to count the number of subword occurrences in circular words. The former takes into account the cyclic structure of a circular word  while the latter is based on the fuzzy perspective that a circular word can be any word in the associated conjugacy class with a uniform probability.

In Section~\ref{SecParMat}, we have seen that certain properties of Parikh matrices of linear words also hold in the context of circular words, but only up to the ternary alphabet. This shows that the behavior of Parikh matrices of circular words differs significantly from those of linear words. As another example, in the classical setting, the value of each minor of any Parikh matrix is nonnegative \cite[Theorem~6]{MSY04}. However, we surmise that this result cannot be extended to the context of circular words as well.

In the classical context of Parikh matrices for linear words, characterization of $M$\!-equivalence, also known as the injectivity problem, has been open for two decades and it is well known that finite applications of Rule~$E_1$ and Rule~$E_2$ do not capture $M$\!-equivalence for the ternary alphabet. Analogously, this is true for Rule~$CE_1$ and Rule~$CE_2$ as well. For example, the two circular words $[aaaacbbc]$ and $[aaacbabc]$ are $M$-equivalent but neither Rule~$CE_1$ nor Rule~$CE_2$ can be applied to either of them. Thus a possible future work would be to develop more rewriting rules that can justify to a better extent--if not completely characterize--$M$\!-equivalence of circular words. For example, one can consider an analogue of the natural generalization of Rule $E_2$ studied in \cite{wT16}.


\end{document}